\documentclass[reqno,11pt]{amsart}

\usepackage{amssymb}
\usepackage{mathrsfs}
\usepackage{cite}
\usepackage{amsfonts}

\newcommand{\sji}{\textsf{sji}}
\def\module#1{\mathrm{mod}\text{-}#1}

\def\pv#1{\ensuremath{{\bf#1}}}
\def\Ann#1{\mathrm{Ann}_S(#1)}

\def\inv{^{-1}}
\def\p{\varphi}

\def\J{\mathrel{{\mathscr J}}} 

\def\R{\mathrel{{\mathscr R}}} 
\def\eL{\mathrel{{\mathscr L}}} 
\def\H{\mathrel{{\mathscr H}}} 

\def\e<{\leq _{E}}

\def\til#1{\ensuremath{\widetilde {#1}}}

\def\malce{\mathbin{\hbox{$\bigcirc$\rlap{\kern-9.3pt\raise0,50pt\hbox{$\mathtt{m}$}}}}}

\def\Ind{\ensuremath{{\bf Ind}}}

\def\1sk{^{(1)}}

\def\to{\rightarrow}

\def\Jnotup#1{{#1}^{\not\hskip 1pt\uparrow}}
\def\Hom{\mathrm{Hom}}

\def\Ind{\mathrm{Ind}}
\def\Coind{\mathrm{Coind}}

\def\Thmname{Theorem}
\def\Propname{Proposition}
\def\Lemmaname{Lemma}
\def\Definitionname{Definition}
\newtheorem{Thm}{\Thmname}[section]
\newtheorem{Prop}[Thm]{\Propname}
\newtheorem{Fact}[Thm]{Fact}
\newtheorem{Lemma}[Thm]{\Lemmaname}
{\theoremstyle{definition}
\newtheorem{Def}[Thm]{\Definitionname}}
{\theoremstyle{remark}
\newtheorem{Rmk}[Thm]{Remark}}
\newtheorem{Cor}[Thm]{Corollary}
{\theoremstyle{remark}
\newtheorem{Example}[Thm]{Example}}

\numberwithin{equation}{section}

\title{Representation Theory of Finite Semigroups over Semirings}

\author[Z.~Izhakian]{Zur Izhakian}
\address{Department of Mathematics\\ Bar-Ilan University\\ Ramat-Gan 52900\\ Israel}
\thanks{The first  author was supported in part by the ISF (Ref. No.
448/09)}
\email{zzur@math.biu.ac.il}

\author[J.~Rhodes]{John Rhodes}
\address{Department of Mathematics\\ University of California, Berkeley\\ Berkeley, CA
94720-3840\\ USA}
\email{rhodes@math.berkeley.edu}\email{BlvdBastille@aol.com}

\author[B.~Steinberg]{Benjamin Steinberg}
\address{Carleton University \\
1125 Colonel By Drive\\
Ottawa, Ontario  K1S 5B6 \\
Canada}
\thanks{The third author was supported in part by NSERC}
\email{bsteinbg@math.carleton.ca}

\date{\today}


\begin{document}
\begin{abstract}
We develop the representation theory of a finite semigroup over an arbitrary commutative semiring with unit, in particular classifying the irreducible and minimal representations. The results for an arbitrary semiring are as good as the results for a field.  Special attention is paid to the boolean semiring, where we also characterize the simple representations and introduce the beginnings of a character theory.
\end{abstract}
\maketitle

\section{Introduction}
This paper is about the representation theory of finite semigroups over commutative semirings with unit.  See~\cite{qtheor} for a modern presentation of finite semigroup theory, including a theory of semirings influenced by a semigroup perspective.

The importance of semirings in mathematics and theoretical computer science was first recognized by Sch\"utz\-en\-ber\-ger~\cite{Schutzsemiring} in his theory of weighted automata and rational power series~\cite{BerstelReutenauer,Salomaa}. Conway also heavily employed semirings in his approach to automata theory~\cite{Conway}. See~\cite{Polak1,Polak2,Polak3,Polak4} for further applications of semirings to theoretical computer science.  Recently, idempotent semirings have entered into mainstream mathematics because they are at the heart of the relatively new subject of tropical geometry~\cite{Sturmfels,Gathmann,Itenberg,Develin,Litvinov}.  They also play a role in the notion of characteristic one being developed by Connes \textit{et al.}~\cite{Connescharone,F1}.

Before turning to semirings, we discuss the classical case of fields.  Work of Clifford~\cite{Clifford1,Clifford2}, Munn~\cite{Munn1,Munn2} and Ponizovski{\u\i}~\cite{Poni} parameterized the irreducible
representations of a finite semigroup over a field in terms of the irreducible
representations of its maximal subgroups.  The reader is referred to~\cite[Chapter 5]{CP}
for a full account of this work.  Explicit constructions of
the irreducible representations were later obtained independently by
Rhodes and Zalcstein~\cite{RhodesZalc} (which was written in 1968) and by Lallement and
Petrich~\cite{LallePet} in terms of the Sch\"utzenberger representation by
monomial matrices~\cite{Schutzmonomial}; see also~\cite{McAlisterCharacter}.  All of these approaches
make use of Rees's Theorem~\cite{Rees} characterizing $0$-simple semigroups up to
isomorphism and Wedderburn theory.  This viewpoint makes it hard to see how to generalize things to rings, let alone semirings.  So it was once believed that we knew everything about representations over the field of complex numbers and nothing about representations over, say, the boolean semiring.

The purpose of this paper is to show that we essentially know just as much about boolean representations as we do about complex ones.  In particular, we construct all irreducible representations of a finite semigroup $S$ over any commutative semiring with unit modulo the case of groups.  In the case of the boolean semiring we identify the congruence associated to the direct sum of all irreducible representations with one of the congruences from the semilocal theory~\cite[Chapter 4]{qtheor}, generalizing the case of representations over fields~\cite{Rhodeschar,AMSV}.  In addition to classifying irreducible representations, we also describe all the minimal representations.  For the case of boolean representations we are also able to classify the simple representations by showing that they are dual to the minimal ones.  In the boolean setting, we propose a notion of characters.

Here we follow the coordinate-free approach from~\cite{myirreps}, which was used to handle  arbitrary commutative rings, in order to deal with semirings. Basic facts about semirings can be found in~\cite{semirings1,semirings2,semirings} or~\cite[Chapters 8 and 9]{qtheor}.

\section{Preliminaries}

We collect here some basic definitions and facts concerning finite
semigroups that can be found in any of~\cite{Arbib,CP,qtheor}.  Let
$S$ be a (fixed) finite semigroup.  If $e$ is an idempotent, then
$eSe$ is a monoid with
identity $e$; its group of units $G_e$ is called the \emph{maximal
  subgroup} of $S$ at $e$.  Two idempotents $e,f$ are said to be
\emph{isomorphic} if there exist $x\in eSf$ and
$x'\in fSe$ such that $xx'=e, x'x=f$.  In this case one can show that
$eSe$ is isomorphic to $fSf$ as
monoids and hence $G_e\cong G_f$; see Fact~\ref{idempotentsiso}.

If $s\in S$, then $J(s) = S^1sS^1$ is the principal (two-sided) ideal
generated by $s$ (here $S^1$ means $S$ with an adjoined identity).
Following Green~\cite{Green}, two elements $s,t$ of a semigroup $S$
are said to be \emph{$\J$-equivalent} if $J(s)=J(t)$.  In this case
one writes $s\J t$.   In fact, there is a preorder on $S$ given by
$s\leq_{\J} t$ if $J(s)\subseteq J(t)$.  This preorder induces an
ordering on $\J$-classes in the usual way.

\begin{Fact}\label{idempotentsiso}
In a finite semigroup, idempotents $e,f$ are isomorphic if and only if $e\J f$, that is, $SeS=SfS$.
\end{Fact}

An element $s$ of a semigroup $S$ is said to be (von Neumann)
\emph{regular} if $s=sts$ for some $t\in S$.  Each idempotent is of
course regular.

\begin{Fact}\label{regularJclass}
Let $S$ be a finite semigroup and $J$ a $\J$-class of $S$.  Then the
following are equivalent:
\begin{enumerate}
\item $J$ contains an idempotent;
\item $J$ contains a regular element;
\item all elements of $J$ are regular;
\item $J^2\cap J\neq \emptyset$.
\end{enumerate}
\end{Fact}

A $\J$-class satisfying the equivalent conditions in
Fact~\ref{regularJclass} is called a \emph{regular} $\J$-class.  The
poset of regular $\J$-classes is denoted $\mathscr U(S)$.   The following standard result from finite semigroup theory will play a role in this paper.

\begin{Fact}\label{dropoutofJ}
Let $S$ be a finite semigroup and $J$ a regular $\J$-class. Let $e\in
J$ be an idempotent.  Then $eSe\cap J=G_e$.
\end{Fact}

Let $J$ be a $\J$-class of $S$.  Set $\Jnotup J = \{s\in S\mid J\nsubseteq
J(s)\}$; it is the ideal of all elements of $S$
that are not $\J$-above some (equals any) element of $J$.

The reader is referred to~\cite[Chapter 9]{qtheor} for basic notions concerning
modules\footnote{We note that some authors use the term semimodule for what we call a module.} over semirings.
Let us say that a congruence on a module over a semiring is \emph{proper} if it has more than one equivalence class and that it is \emph{trivial} if the associated partition is into singletons.  Fix a base commutative semiring ring $k$ with unit.

\begin{Def}
If $A$ is a $k$-algebra,
not necessarily unital, then a right $A$-module $M$ with $MA\neq 0$ is said to be:
\begin{enumerate}
 \item \emph{simple} if it has only trivial quotient modules;
\item \emph{minimal} if it contains no proper non-zero submodules;
\item \emph{irreducible} if it is both simple and minimal.
\end{enumerate}
\end{Def}

Over a ring $k$, all three notions coincide.
Observe that minimality amounts to asking that, for all $0\neq
m\in M$, the cyclic module $mA$ is $M$.

 The category of (right) $A$-modules will be denoted $\module A$.
We adopt here the convention that if $A$ is unital, then by $\module A$ we
mean the category of unitary $A$-modules.  The reader should verify
that all functors considered in this paper respect this convention.

\begin{Def}[Radical and socle]\label{radicaldef}
Let us denote by $\mathrm{rad}(M)$ the intersection of all maximal congruences on a module $M$; it is called the \emph{radical} of $M$.  The submodule of $M$ generated by all minimal submodules is denoted $\mathrm{Soc}(M)$ and is called the \emph{socle} of $M$.
\end{Def}

We state and prove Schur's lemma in the semiring context.

\begin{Lemma}[Schur]
Let $M$ and $N$ be $A$-modules.  Then each non-zero homomorphism $\p\colon M\to N$ is:
\begin{enumerate}
 \item injective if $M$ is simple;
\item surjective if $N$ is minimal;
\item an isomorphism if $M$ is simple and $N$ is minimal.
\end{enumerate}
\end{Lemma}
\begin{proof}
Suppose that $\p\neq 0$.  Then $\p(M)$ is a non-zero submodule of $N$ and $\ker \p$ is a proper congruence on $M$, from which the result is trivial.
\end{proof}

The class of minimal modules is closed under quotients.

\begin{Prop}\label{minishered}
Let $M$ be a minimal $A$-module. Then any non-zero quotient module  of $M$ is also minimal.
\end{Prop}
\begin{proof}
Suppose that $M$ is minimal and $\p\colon M\to N$ is a surjective $A$-module homomorphism.  Let $N'$ be a non-zero submodule of $N$.  Then $\p\inv(N)$ is a non-zero submodule of $M$ and so $M=\p\inv (N)$, whence $N=N'$.  Thus $N$ is minimal.
\end{proof}

\section{Construction of the irreducible and minimal modules}
Fix a finite semigroup $S$ and a commutative semiring with unit $k$.  The
semigroup algebra $kS$ need not be unital.  If $M$ is a $kS$-module, then $\Ann M=\{s\in S\mid Ms=0\}$ is an ideal of $S$.
The following definition, due to Munn~\cite{Munn1}, is crucial to semigroup representation
theory.

\begin{Def}[Apex]
A regular $\J$-class $J$ is said to be the \emph{apex} of a $kS$-module $M$
if $\Ann M=\Jnotup J$.
\end{Def}
It is easy to see that $M$ has apex $J$ if and only if $J$ is the
unique $\leq_{\J}$-minimal $\J$-class that does not annihilate $M$.  The notion of apex is closely related to that of lifting $\J$-classes~\cite[Lemma 4.6.10]{qtheor}.  For further background see~\cite[Chapter 4.6]{qtheor}.

Fix an idempotent transversal $E=\{e_J\mid J\in \mathscr U(S)\}$ to
the set $\mathscr U(S)$ of regular $\J$-classes  and set $G_J=
G_{e_J}$.  Let $A_J=kS/k\Jnotup J$.  Notice that the category of
$kS$-modules with apex $J$ can be identified with the full subcategory
of $\module {A_J}$ whose objects are modules $M$ such that
$Me_J\neq 0$.

Our first goal is to show that every minimal module has an apex.  This
result is due independently to Munn and
Ponizovski{\u\i}~\cite{Munn1,Munn2,Poni} in the case of fields.

\begin{Thm}\label{haveapex}
Let $M$ be a minimal $kS$-module.  Then $M$ has an apex.
\end{Thm}
\begin{proof}
Because $MkS\neq 0$, there is a $\leq_{\J}$-minimal $\J$-class $J$ such
that $J\nsubseteq \Ann M$.  Let $I = S^1JS^1$; of course, $I$ is an
ideal of $S$. Since $I\setminus J$ annihilates $M$
by minimality of $J$, it follows $0\neq MkJ=MkI$. From the fact that $I$ is an
ideal of $S$, we may deduce that $MkI$ is a $kS$-submodule and so  by minimality
\begin{equation}\label{apex1}
M= MkI = MkJ.
\end{equation}
Therefore, since $J\Jnotup J\subseteq I\setminus J\subseteq \Ann M$, it follows
from \eqref{apex1} that $\Jnotup J= \Ann M$.  Now if $J$ is not regular,
then Fact~\ref{regularJclass} implies $J^2\subseteq I\setminus J$ and hence $J$
annihilates $M$ by \eqref{apex1}, a contradiction.  Thus $J$ is
regular and is an apex for $M$.
\end{proof}

Now we wish to establish a bijection between irreducible $kS$-modules with apex
$J$ and irreducible $kG_J$-modules.  This relies on a semiring analogue of a well-known result of
Green~\cite[Chapter 6]{Greenpoly}.   Let $A$ be a $k$-algebra and $e$ an idempotent
of $A$. Then $eA$ is an $eAe$-$A$-bimodule and $Ae$ is an $A$-$eAe$-bimodule.  Hence we have a
restriction functor $\mathrm{Res}\colon \module A\to \module {eAe}$ and induction/coinduction
functors $\mathrm{Ind},\mathrm{Coind}\colon \module {eAe}\to \module A$ given by
\begin{equation*}
\mathrm{Ind}(M) = M\otimes _{eAe}eA,\ \mathrm{Res}(M) = Me\ \text{and}\
\mathrm{Coind}(M) = \mathrm{Hom}_{eAe}(Ae,M).
\end{equation*}
Moreover,  $\mathrm{Ind}$ is right exact, $\mathrm{Res}$ is exact,  $\mathrm{Coind}$ is left exact and
$\mathrm{Ind}$ and $\mathrm{Coind}$ are the left and right adjoints of $\mathrm{Res}$, respectively (where a functor is left exact if it preserves finite limits and right exact if it preserves finite colimits).   This follows from observing that
$\mathrm{Hom}_A(eA,M) = Me = M\otimes_A Ae$ and the
usual adjunction between hom and the tensor product. Furthermore, it is
well known that the unit of the first adjunction gives a natural isomorphism
$M\cong \mathrm{Ind}(M)e$ while the counit of the second gives a
natural isomorphism $\mathrm{Coind}(M)e\cong M$.

We need the following lemma relating congruences on $M$ and $Me$.

\begin{Prop}\label{extendcongruence}
Let $M$ be an $A$-module and let $e\in A$ be an idempotent.  Then every congruence $R$ on $Me$ extends to a congruence $R'$ on $M$ by setting $m\mathrel {R'} m'$ if and only if $mae\mathrel R m'ae$ for all $a\in A$.  Moreover, $R'$ is the largest congruence whose restriction to $Me$ is contained in $R$.
\end{Prop}
\begin{proof}
Clearly $R'$ is $k$-module congruence. To see that it is an $A$-module congruence, suppose $b\in A$ and $m\mathrel{R'} m'$.  Then for any $a\in A$, we have $m(bae)
\mathrel R m'(bae)$ and so $mb\mathrel{R'} m'b$ as required.  Finally, observe that if $m,m'\in Me$ and $m\mathrel{R'} m'$, then $m=mee\mathrel R m'ee=m'$.  Conversely, if $m,m'\in Me$ are $R$-related and $a\in A$, then $mae=meae\mathrel R m'eae=m'ae$ and so $m\mathrel {R'} m'$.  Thus $R'$ extends $R$.

Finally, suppose that $R''$ is some other congruence on $M$ whose restriction to $Me$ is contained in $R$ and suppose $m\mathrel{R''} m'$.  Then if $a\in A$ one has $mae\mathrel R m'ae$ and so $m\mathrel{R'} m'$ completing the proof.
\end{proof}

Let $M$ be an $A$-module and define
\[N(M) = \{(m,m')\mid mae=m'ae, \forall a\in A\}\] and $L(M) = MeA$.  It is easily verified using Proposition~\ref{extendcongruence} that $N(M)$ is
the largest congruence on $M$ whose restriction to $Me$ is trivial.  On the other hand, $L(M)$ is the smallest $A$-submodule of $M$ with $L(M)e=Me$.  The congruence class of $0$ under $N(M)$ is  \[K(M) = \{m\in M\mid mae=0, \forall a\in A\}.\]  It is also the largest submodule of $M$ annihilated by $e$.  The constructions $K,L,N$ are functorial.

Observe that if $V$ is an $eAe$-module, then
\begin{equation}\label{Indspan}
L(\mathrm{Ind}(V))=\mathrm{Ind}(V)eA =
V\otimes_{eAe}eAeA = V\otimes _{eAe}eA = \mathrm{Ind}(V).
\end{equation}
Alternatively, one can observe that $\mathrm{Ind}(V)eA$ satisfies the same universal property as $\mathrm{Ind}(V)$.

The analogue of our next result for rings can be found in~\cite[Chapter 6.2]{Greenpoly}.  The proof for semirings is essentially a modification.  We remind the reader that if $M$ is an $A$-module and $N$ is a submodule, then $M/N$ is the quotient of $M$ by the congruence given by $m\equiv m'\bmod N$ if and only if there exist $n,n'\in N$ with $m+n=m'+n'$; it is the largest congruence identifying all elements of $N$ with $0$, cf.~\cite[Chapter 9]{qtheor}.

\begin{Lemma}\label{greenslemma}
Let $A$ be a $k$-algebra and $e$ an idempotent of $A$.
\begin{enumerate}
\item If $M$ is an irreducible (minimal, simple) $A$-module, then $Me=0$ or $Me$ is an irreducible (minimal, simple)  $eAe$-module.
\item If $V$ is a minimal $eAe$-module, then the unique maximal $A$-submodule of $\mathrm{Ind}(V)$ is $K(\mathrm{Ind}(V))$.  Moreover, the minimal $A$-modules $M$ with $Me\cong V$ are (up to isomorphism) the quotients of $\mathrm{Ind}(V)$ by congruences in between congruence modulo $K(\mathrm{Ind}(V))$ and $N(\mathrm{Ind}(V))$.
\item  If $V$ is a simple $eAe$-module, then
  $\mathrm{Ind}(V)$ has unique maximal
  congruence  $N(\mathrm{Ind}(V))$.  Moreover, if $V$ is irreducible, then the quotient
  $\mathrm{Ind}(V)/N(\mathrm{Ind}(V))$ is the unique irreducible
$A$-module $M$ with $Me\cong V$.
\item  If $V$ is a minimal $eAe$-module, then $\mathrm{Coind}(V)$ has unique minimal $A$-submodule $L(\mathrm{Coind}(V))$.  Furthermore, if $V$ is irreducible, then $L(\mathrm{Coind}(V))$ is the unique irreducible $A$-module $M$ with $Me\cong V$.
\end{enumerate}
Consequently, restriction yields a bijection between irreducible
  $A$-modules that are not annihilated by $e$ and irreducible
  $eAe$-modules.
\end{Lemma}
\begin{proof}
To prove (1), assume $Me\neq 0$.  Suppose first that $M$ is minimal and let $m\in
Me$ be non-zero.  Then $meA = mA=M$, so $meAe=Me$. Thus $Me$ is minimal.  Next suppose that $M$ is simple and let $R$ be a proper congruence on $Me$.  Then $R'$ from Proposition~\ref{extendcongruence} is a proper congruence on $M$ and hence trivial.  But then $R$, which is the restriction of $R'$ to $Me$, is trivial.  Thus $Me$ is simple. The irreducible case follows by combining the minimal and simple cases.

Next we turn to (2).  Suppose that $L$ is a proper submodule of $\mathrm{Ind}(V)$.  Then $Le$ is an $eAe$-submodule of $\mathrm{Ind}(V)e\cong V$.  By minimality of $V$, either $Le=\{0\}$, and so $L\subseteq K(\mathrm{Ind}(V))$, or $Le=\mathrm{Ind}(V)e$.  In the latter case we have $L\supseteq LeA=\mathrm{Ind}(V)eA=\mathrm{Ind}(V)$, where the last equality uses \eqref{Indspan}.  Thus $K(\mathrm{Ind}(V))$ is the unique maximal proper submodule.  It follows that $\mathrm{Ind}(V)/K(\mathrm{Ind}(V))$ is minimal.  Next observe that since $N(\mathrm{Ind}(V))$ has trivial restriction to $\mathrm{Ind}(V)e$, Proposition~\ref{minishered} now implies that, for any congruence $R$ between congruence modulo $K(\mathrm{Ind}(V))$ and $N(\mathrm{Ind}(V))$, one has that $\mathrm{Ind}(V)/R$ is minimal and $[\mathrm{Ind}(V)/R]e\cong V$.  Finally, suppose that $M$ is a minimal $A$-module with $Me\cong V$.  Then the adjunction yields a non-zero homomorphism $\p\colon \mathrm{Ind}(V)\to M$ which is surjective by minimality and restricts to an isomorphism of $eAe$-modules from $\mathrm{Ind}(V)e\to Me$. In particular, it follows that the congruence $\ker \p$ is injective on $\mathrm{Ind}(V)e$ and hence contained in $N(\mathrm{Ind}(V))$.  Also, as $K(\mathrm{Ind}(V))e=0$ and $Me\neq 0$,  minimality implies $\p(K(\mathrm{Ind}(V)))=0$.  This establishes (2).

Now we turn to (3).  Let $V$ be a simple $eAe$-module.  It is immediate from \eqref{Indspan} that
any proper congruence on $\Ind(V)$ must restrict to a proper congruence on $\mathrm{Ind}(V)e\cong V$.  But $V$ is simple and so the only proper congruence on $V$ is the trivial one.  Proposition~\ref{extendcongruence} now implies that $N(\Ind(V))$ is the unique maximal congruence on $\Ind(V)$.  Assume furthermore that $V$ is irreducible.  Then $\Ind(V)/N(\Ind(V))$ is minimal, and hence irreducible, and $[\Ind(V)/N(\Ind(V))]e\cong V$ by (2).
Uniqueness follows also from (2), since if $M$ is irreducible with $Me\cong V$, then in particular $M$ is minimal and so isomorphic to $\Ind(V)/R$ for a certain congruence $R\subseteq N(\Ind(V))$.  But simplicity of $M$ implies that $R=N(\Ind(V))$.  This completes the proof of (3).

Finally, to prove (4) first observe that if $M$ is any non-zero
$A$-submodule of $\mathrm{Coind}(V)$, then $Me\neq 0$.  Indeed,
suppose $Me=0$ and let $\p\in M$. Then, for any $x$ in $Ae$, we have
$\p(x) = (\p xe)(e)=0$ since $\p xe\in Me=0$.  It follows that $M=0$.
Since $\mathrm{Coind}(V)e\cong V$ is a minimal $eAe$-module, it now
follows that if $M$ is a non-zero $A$-submodule of
$\mathrm{Coind}(V)$, then $Me=\mathrm{Coind}(V)e$ and hence
\[L(\mathrm{Coind}(V)) = \mathrm{Coind}(V)eA\subseteq MeA\subseteq
M.\]  This establishes that $L(\mathrm{Coind}(V))$ is the unique
minimal $A$-submodule.

Suppose now that $V$ is irreducible.  We must show that any proper congruence on $L(\mathrm{Coind}(V))$ is trivial.  Since $\mathrm{Coind}(V)e\cong V$ is irreducible, it follows that any proper congruence on $L(\mathrm{Coind}(V))$ is trivial on $L(\mathrm{Coind}(V))e$ and so by Proposition~\ref{extendcongruence} it suffices to show that $N(L(\mathrm{Coind}(V)))$ is the trivial congruence.  So suppose $\p,\psi\in L(\Coind(V))$ are equivalent and let $a\in A$.  Then \[\p(ae) = (\p ae)(e) = (\psi ae)(e) = \psi(ae)\] and hence $\p=\psi$.  This completes the proof $L(\mathrm{Coind(V)})$ that is irreducible.

Since $L(\mathrm{Coind}(V))e =
\mathrm{Coind}(V)eAe = \mathrm{Coind}(V)e\cong V$, it just remains to
prove uniqueness.  Suppose $M$ is an irreducible $A$-module with $Me\cong
V$. Then the existence of a non-zero element of
$\mathrm{Hom}_{eAe}(Me,V)\cong \mathrm{Hom}_A(M,\mathrm{Coind}(V))$
implies that $M$ admits a non-zero homomorphism to
$\mathrm{Coind}(V)$.  Hence $M$ is isomorphic to an irreducible (and hence minimal)
$A$-submodule of $\mathrm{Coind}(V)$. But $L(\mathrm{Coind}(V))$ is
the unique minimal submodule of $\mathrm{Coind}(V)$ and so $M\cong
L(\mathrm{Coind}(V))$, as required.
\end{proof}

We may now etablish an analogue of the Clifford-Munn-Ponizovski{\u\i} theorem for semirings.

\begin{Thm}\label{CMP}
Let $S$ be a finite semigroup, $k$ a commutative semiring with unit
and $E=\{e_J\mid
J\in \mathscr U(S)\}$ be an idempotent transversal to the set $\mathscr
U(S)$ of regular
$\J$-classes of $S$.  Let $G_J$ be the maximal subgroup $G_{e_J}$.
Define functors $\mathrm{Ind}_J,\mathrm{Coind}_J\colon \module kG_J\to \module kS$ by
\begin{align*}
\mathrm{Ind}_J(V)&= V\otimes_{kG_J}e_J(kS/k\Jnotup J)\\
\mathrm{Coind}_J(V)&= \mathrm{Hom}_{kG_J}((kS/k\Jnotup J)e_J,V).
\end{align*}
Then:
\begin{enumerate}
\item If $M$ is an irreducible (minimal) $kS$-module with apex $J$, then $Me_J$ is an
  irreducible (minimal) $kG_J$-module;
\item If $V$ is an irreducible $kG_J$-module, then \[N = \{(u,v)\in
  \mathrm{Ind}_J(V)\mid uae=vae, \forall a\in kS/k\Jnotup J\},\] is the unique maximal
  congruence on $\mathrm{Ind}_J(V)$ and
  $\mathrm{Ind}_J(V)/N$ is the unique irreducible $kS$-module $M$
with apex $J$ such that $Me_J\cong V$;
\item If $V$ is an irreducible $kG_J$-module, then the unique minimal $kS$-submodule of
  $\mathrm{Coind}_J(V)$ is
  $\mathrm{Coind}_J(V)e_JkS$, which moreover is the unique irreducible
  $kS$-module $M$ with apex $J$ such that $Me_J\cong V$.
\item  If $V$ is a minimal $kG_J$-module, then \[K= \{m\in
  \mathrm{Ind}_J(V)\mid mae=0, \forall a\in kS/k\Jnotup J\}\] is the unique maximal
  $kS$-submodule of $\mathrm{Ind}_J(V)$. Moreover, the minimal $kS$-modules $M$ with apex $J$ such that $Me_J\cong V$ are up to isomorphism the quotients $\mathrm{Ind}_J(V)/R$ with (retaining the notation of (2)) $R\subseteq N$ and $K$ contained in a single class of $R$.
\end{enumerate}
Consequently, there is a bijection between the irreducible $kS$-modules and the irreducible $kG_J$-modules, where $J$ runs over
$\mathscr U(S)$.
\end{Thm}
\begin{proof}
Theorem~\ref{haveapex} implies that every minimal $kS$-module $M$ has an
apex. Again setting $A_J=kS/k\Jnotup J$ for a regular $\J$-class $J$, we know
that irreducible $kS$-modules with apex $J$ are in bijection with irreducible
$A_J$-modules $M$ such that $Me_J\neq 0$. It follows directly from
Fact~\ref{dropoutofJ} that \[e_JA_Je_J=ke_JSe_J/ke_J\Jnotup Je_J=kG_J.\]
Lemma~\ref{greenslemma} then yields that irreducible $A_J$-modules not
annihilated by $e_J$, that is, irreducible $kS$-modules with apex $J$, are
in bijection with irreducible $kG_J$-modules in the prescribed manner.  Similarly, the minimal $kS$-modules are as advertised by application of Lemma~\ref{greenslemma}.
\end{proof}

\subsection{A construction in coordinates}
Let us relate the above construction of the irreducible
modules  to the explicit ones found in~\cite{RhodesZalc,LallePet} for fields.   All the
facts about finite semigroups used in this discussion can be found in the
appendix of~\cite{qtheor} or in~\cite{Arbib}. According to
Green~\cite{Green}, two elements $s,t$ of a semigroup are said to be
\emph{$\mathscr R$-equivalent} if they generate the same principal right
ideal.  Dually $s,t$ are said to be \emph{$\eL$-equivalent} if they generate
the same principal left ideal.

Once more let $S$ be a finite semigroup, $k$ a commutative semiring with unit
and $E=\{e_J\mid
J\in \mathscr U(S)\}$ an idempotent transversal to the set $\mathscr
U(S)$ of regular
$\J$-classes of $S$.  Let $G_J$ be the maximal subgroup $G_{e_J}$.
We use $L_J$ and $R_J$ for the $\eL$- and $\R$-classes of $e_J$, respectively.

Here we follow~\cite{rrbg} to give a concrete description of the irreducible $kS$-modules.  If $V$ is an irreducible $kG_J$-module, we use $\til V$ for the corresponding irreducible $kS$-module.  The reader should recall the definition of radical and socle from Definition~\ref{radicaldef}.

\begin{Prop}\label{computeirreducible}
Let $V$ be an irreducible $kG_J$-module.  Then there is a natural isomorphism $\Hom_{kS}(\Ind_J(V),\Coind_J(V))\cong \Hom_{kG_J}(V,V)\neq 0$.  Moreover, if $\p\in  \Hom_{kS}(\Ind_J(V),\Coind_J(V))$ is non-zero, then we have $\ker \p = \mathrm{rad}(\Ind_J(V))$ and $\mathop{\mathrm{Im}}\p = \til V = \mathrm{Soc}(\Coind_J(V))$.
\end{Prop}
\begin{proof}
First note that since $\Ind_J(V),\Coind_J(V)$ are $kS/k\Jnotup J$-modules, the adjunction yields \[\Hom_{kG_J}(V,V)=\Hom_{kG_J}(\Ind_J(V)e_J,V)\cong  \Hom_{kS}(\Ind_J(V),\Coind_J(V)).\]  Suppose now that $\p\colon \Ind_J(V)\to \Coind_J(V)$ is a non-zero homomorphism.
Because $\Ind_J(V)e_JkS=\Ind_J(V)$ by construction, it follows that \[\p(\Ind_J(V)) = \p(\Ind_J(V))e_JkS\subseteq \Coind_J(V)e_JkS=\til V.\]  Since $\p\neq 0$, it follows by irreducibility of $\til V$, that $\p(\Ind_J(V))=\til V$.  As $\Ind_J(V)$ has a unique maximal congruence, we conclude that $\ker \p = \mathrm{rad}(\Ind_J(V))$.
\end{proof}

Observe that as $k$-modules, $kL_J = (kS/k\Jnotup{J})e_J$ and $kR_J = e_JkS/k\Jnotup{J}$ by stability.  Moreover, the corresponding $kG_J$-$kS$-bimodule structure on $kR_J$ is induced by left multiplication by elements of $G_J$ and by the right Sch\"utzenberger representation of $S$ on $R_J$~\cite{CP,Arbib,qtheor}  (i.e., the action of $S$ on $R_J$ by partial functions obtained via restriction of the regular action).  To simplify notation, we will use $kR_J$ and $kL_J$ for the rest of this section.  Then we have
\begin{align*}
\Ind_J(V) &= V\otimes_{kG_J} kR_J \\
\Coind_J(V) &= \Hom_{kG_J}(kL_J,V).
\end{align*}

Multiplication in the semigroup induces a non-zero homomorphism
\begin{equation*}\label{bilinearmap}
C_J\colon kR_J\otimes_{kS} kL_J\cong e_JkS/k\Jnotup{J}\otimes_{kS} (kS/k\Jnotup{J})e_J\to e_J(kS/k\Jnotup{J})e_J \cong kG_J
\end{equation*}
which moreover is a map of $kG_J$-bimodules.

Let $T\subseteq R_J$ be a complete set of representatives of the $\eL$-classes of $J$ and $T'\subseteq L_J$ be a complete set of representatives of the $\R$-classes of $J$.  Then $G_J$ acts freely on the left of $R_J$ and $T$ is a transversal for the orbits and dually $T'$ is a transversal for the orbits of the free action of $G_J$ on the right of $L_J$, see~\cite[Appendix A]{qtheor}.  Thus $kR_J$ is a free left $kG_J$-module with basis $T$ and $kL_J$ is a free right $kG_J$-module with basis $T'$.   It is instructive to verify that the associated matrix representation over $kG_J$ of $S$ on $kR_J$ is the classical right Sch\"utzenberger representation by row monomial matrices and the representation of $S$ on $kL_J$ is the left Sch\"utzenberger representation by column monomial matrices~\cite{CP,qtheor,Arbib,RhodesZalc}.    Hence if $\ell_J=|T|$ and $r_J=|T'|$, then as $kG_J$-modules we have $kR_J\cong kG_J^{\ell_J}$ and $kL_J\cong kG_J^{r_J}$.
 Thus $C_J$ is the bilinear form given by the $\ell_J\times r_J$-matrix (also denoted $C_J$) with
\begin{equation}\label{structurematrix}
(C_J)_{ba} = \begin{cases} \lambda_b\rho_a & \lambda_b\rho_a\in J\\ 0 & \text{otherwise}\end{cases}
\end{equation}
where $\lambda_b\in T$ represents the $\eL$-class $b$ and $\rho_a\in T'$ represents the $\R$-class $a$.  Note that $(C_J)_{ba}\in G_J\cup \{0\}$ by stability and $C_J$ is just the usual sandwich (or structure) matrix of the $\J$-class $J$ coming from the Green-Rees structure theory~\cite{CP,Arbib,qtheor}.  The reader may take \eqref{structurematrix} as the definition of the sandwich matrix if he/she so desires.

Suppose now that $V$ is a $kG_J$-module. We can consider the induced map
\begin{equation*}\label{bilinearmap2}
V\otimes C_J\colon V\otimes_{kG_J}kR_J\otimes_{kS} kL_J\to V\otimes_{kG_J}kG_J\cong V
\end{equation*}
which moreover is non-zero as $v\otimes e_J\otimes e_J\mapsto v$.
From the isomorphism \[\Hom_{kG_J}(V\otimes_{kG_J}kR_J\otimes_{kS} kL_J,V)\cong \Hom_{kS}(V\otimes_{kG_J} kR_J,\Hom_{kG_J}(kL_J,V))\] we obtain a corresponding non-zero $kS$-linear map \[V\otimes C_J\colon \Ind_J(V)\to \Coind_J(V)\] (abusing notation).

Putting together the above discussion with Proposition~\ref{computeirreducible}, we obtain the following result.
\begin{Thm}\label{RZirreducible}
Let $V$ be an irreducible $kG_J$-module.  Then the irreducible $kS$-module corresponding to $V$ is the image of the morphism \[V\otimes C_J\colon \Ind_J(V)\to \Coind_J(V)\] where $C_J$ is the sandwich matrix for $J$, i.e., it is the $k$-span of the rows of $V\otimes C_J$.
\end{Thm}

\begin{Rmk}
Note that since $kR_J$ and $kL_J$ are free $kG_J$-modules with bases $T$ and $T'$ respectively, as $k$-modules, we have $\Ind_J(V) = V^{\ell_J}$ and $\Ind_J(V)=V^{r_J}$; in particular, these two functors are exact.  Moreover, one can easily compute that $V\otimes C_J$ is given via right multiplication by $C_J$ where we view elements of $V^{\ell_J}$ and $V^{r_J}$ as row vectors with entries in $V$.
\end{Rmk}

A semigroup $S$ is called \emph{generalized group mapping} if it has a \emph{distinguished} ($0$-)minimal ideal $I$ on which it acts faithfully on both the left and right~\cite[Chapter 4]{qtheor}.  The distinguished ideal is unique and regular.   The following result generalizes a result of Rhodes and Zalcstein for fields~\cite{RhodesZalc}.  The original proof uses Wedderburn theory, while our proof uses the description of the irreducible modules.  Recall that a module $M$ is \textit{faithful} for $S$ if $ms=mt$ for all $m\in M$ implies $s=t$.

\begin{Prop}\label{GGM}
Let $S$ be a finite semigroup and suppose that the irreducible $kS$-module $M$ is faithful.  Then $S$ is generalized group mapping with distinguished ideal $J$ or $J\cup \{0\}$ where $J$ is the apex of $M$.
\end{Prop}
\begin{proof}
We just handle the case that $S$ has a zero element $0$, as the other case is easier.  Also, by Schur's lemma, $0$ must act either as the identity (in which case $S$ is trivial) or as the zero endomorphism of $M$.  Suppose that we are in the latter case.  By definition of an apex, it is clear that $I=J\cup \{0\}$ is an ideal.  Let $e\in E(J)$ and put $V=Me$.  Then $M$ is a quotient of $\Ind_J(V)$ and a submodule of $\Coind_J(V)$.  It then follows that $S$ acts faithfully on $\Ind_J(V)$ and $\Coind_J(V)$.  Now if $s,t$ act the same on the right of $I$, then they will act the same on the right of $R_e$ and hence on $\Ind_J(V)$.  Thus the action of $S$ on the right of $I$ is faithful.   Similarly, if $s,t$ act the same on the left of $I$, then they act the same on the left of $L_e$ and hence on $\Coind_J(V)$.  This completes the proof that $S$ is generalized group mapping.
\end{proof}

\section{The case of idempotent semirings}
It turns out that we could have restricted our attention to two cases for irreducible modules: rings (already handled in~\cite{myirreps}) and idempotent semirings, as the following observation shows, cf.~\cite{Zsemiring}.

\begin{Prop}\label{additivestructure}
Let $S$ be a finite semigroup and $M$ a simple $kS$-module.  Then the additive structure of $M$ is either an abelian group or a join semilattice with minimum.
\end{Prop}
\begin{proof}
Define a relation on $M$ by $m\equiv n$ if there exists $j,k\in \mathbb N$ so that $jm\in n+M$ and $kn\in m+M$.  It is straightforward to verify that $\equiv$ is a congruence.  Suppose first that $\equiv$ is trivial and let $m\in M$.  Then since $m\equiv m+m$, it follows that $m+m=m$ and so $M$ is a join semilattice with minimum.  If $\equiv$ is not proper, then $m\equiv 0$ and hence $0\in M+m$ and so $M$ is an additive group.
\end{proof}

Let us now consider the join semilattice case.  It turns out that in this case, every representation of a finite group is trivial and so there is exactly on irreducible module over an idempotent semiring associated to any regular $\J$-class of a finite semigroup.

\begin{Prop}\label{nogroups}
Let $G$ be a finite group and $k$ a commutative semiring with unit.  Suppose that $M$ is a minimal $kG$-module with idempotent addition.  Then there is a quotient semiring $k'$ of $k$ so that $M=k'$ with trivial action by $G$.
\end{Prop}
\begin{proof}
By minimality, it follows that $G$ acts by automorphisms of $M$.  Let $0\neq m_0\in M$ and put $m=\sum_{g\in G}m_0g$.  Then $m$ is evidentally fixed by $G$ and so the $k$-span of $m$ is a $kG$-submodule and hence is $M$ by minimality.  If $k'$ is the faithful quotient of $k$ acting on $M$, then $M=k'$ with the trivial $G$-action.
\end{proof}

In particular, if $k$ is a congruence-free commutative idempotent semiring with unit (e.g., the boolean semiring $\mathbb B$), then the trivial action of $G$ on $k$ is the only irreducible $kG$-module for $G$ a group.  Recall that a finite semigroup is \emph{aperiodic} if all its maximal subgroups are trivial.

\begin{Cor}
Let $S$ be a finite semigroup with a faithful irreducible $kS$-module $M$ whose addition is idempotent.  Then $S$ is generalized group mapping with aperiodic distinguished minimal ideal.
\end{Cor}
\begin{proof}
By Proposition~\ref{GGM}, $S$ is generalized group mapping.  Let $J$ be the apex of $M$ and let $e\in J$ be an idempotent.  Then $Me$ is an irreducible $kG_e$-module, and hence has trivial action of $G_e$ by Proposition~\ref{nogroups}.  By faithfulness, we conclude $G_e$ is trivial.  Since $G_e$ is the maximal subgroup of the distinguished ideal of $S$, this completes the proof.
\end{proof}

The irreducible $\mathbb BS$-modules admit the following description.

\begin{Thm}\label{desribeirreducibles}
Let $S$ be a finite semigroup.  Then the minimal and irreducible $\mathbb BS$-modules are obtained as follows.  Fix a regular $\J$-class $J$ of $S$ with set $A$ of $\R$-classes and $B$ of $\eL$-classes.  Let $C\colon B\times A\to \mathbb B$ be the matrix with $C_{ba}=1$ if and only if the $\H$-class $a\cap b$ contains an idempotent.  Let $\mathbb BB$ be the free module on $B$ and consider the natural right action of $S$ on $\mathbb BB$.  Define a congruence $\equiv$ on $\mathbb BB$ by putting $m\equiv n$ if $ms=0\iff ns=0$ for all $s\in S$. Then:
\begin{enumerate}
\item The congruence $\equiv$ is the unique maximal proper congruence on $\mathbb B$;
\item The module $\mathbb BB$ is minimal, as is every proper quotient of $\mathbb B$;
\item $\mathbb BB/{\equiv}$ is irreducible and
all irreducible modules $\mathbb BS$-modules are of this form for some $\J$-class.
\end{enumerate}
Alternatively, $\mathbb BB/{\equiv}$ can be identified with the $\mathbb B$-span of the rows of the matrix $C$.
\end{Thm}
\begin{proof}
First observe that by Proposition~\ref{nogroups}, the only minimal (and hence irreducible) module for a finite group $G$ over $\mathbb B$ is $\mathbb B$ equipped with the trivial action of $G$.  Thus each $\J$-class provides a unique irreducible $\mathbb BS$-module $M_J$, coming from the trivial representation of the maximal subgroup.

Let us observe that $C$ is the tensor product of the structure matrix of $J$ with the trivial representation of the maximal subgroup of $G$.  It follows that $M_J$ can be identified with the $\mathbb B$-span of the rows of $C$ by Theorem~\ref{RZirreducible}.  This theorem also implies that $M_J\cong\mathbb BB/{\ker C}$.  Let $m,n\in \mathbb BB$.  Then first observe that $mC$ is determined by which entries are $0$.  Now $(mC)_a = \sum_{b\in B}m_bC_{ba}$ and hence is $0$ if and only if $m$ is annihilated by the $\R$-class $a$.  It now follows that $mC=nC$ if and only if, for all $s\in J$, one has $ms=0\iff ns=0$.

Next observe that since $J$ is the apex of $M_J$, it follows easily by minimality of $M_J$ that if $0\neq m\in M_J$, then $mJ\neq 0$.  Thus if $ms\neq 0$ with $s\in S$, we can find $x\in J$ so that $msx\neq 0$.  In particular, it follows by the definition of an apex that $sx\in J$.  We conclude that $m\equiv n$ if and only if $ms=0\iff ns=0$ for all $s\in J$.  This completes the proof of the (3) of the proposition.  The general theory implies that $\mathbb BB$ has a unique maximal congruence and hence it must be $\equiv$.

It remains to prove (2).  First notice that by the general theory, the unique maximal submodule of $\mathbb BB$ is the congruence class of $0$ under the unique maximal congruence.  Therefore, the unique maximal submodule of $\mathbb BB$ is the set of all vectors annihilated by $C$.  But $C$ is a boolean matrix with no zero rows or columns.  Hence, no non-zero vector in $\mathbb BB$ is annihilated by $C$. Thus $\mathbb BB$ is minimal and hence so are all its quotients by Proposition~\ref{minishered}.
\end{proof}

In~\cite[Chapter 4.6]{qtheor}, it is shown that, for each regular $\J$-class $J$ of a finite semigroup $S$, there is a unique congruence $\equiv_J$ on $S$ such that $S/{\equiv_J}$ is generalized group mapping with aperiodic distinguished ideal and $\Jnotup J$ maps to $0$.  The resulting quotient is denoted $\mathsf{AGGM}_J(S)$ and the quotient map is written $\Gamma_J\colon S\to \mathsf{AGGM}_J(S)$.  Consequently, $\mathsf{AGGM}_J(S)$ must then be isomorphic to the image of $S$ under the irreducible representation $S\to \mathrm{End}_{\mathbb B}(M_J)$ constructed in the above proof.  In~\cite[Chapter 4]{qtheor} (see also~\cite{folleyT}), it is shown that the intersection over all regular $\J$-classes $J$ of the congruences $\equiv_J$ is the largest $\J'$-congruence on $S$.  Recall that a congruence $\equiv$ is a \emph{$\J'$-congruence} if $s\equiv t$ and $s,t$ regular implies $s\J t$.  Thus we have proved:

\begin{Thm}\label{largeJ'}
The largest $\J'$-congruence on a finite semigroup is the congruence associated to the direct sum of all irreducible representations of $S$ over $\mathbb B$.
\end{Thm}

The analogous theorems for fields of characteristic $0$ and $p$ can be found in~\cite{Rhodeschar,AMSV}.

A semigroup $S$ is called a \emph{local group} if $eSe$ is a group for each idempotent $e\in S$.  The collection of finite local groups is denoted $\mathbb L\pv G$ and is a \emph{pseudovariety}, i.e., is closed under finite products, subsemigroups and homomorphic images.  If $\pv V$ is a pseudovariety, then the Mal'cev product $\mathbb L\pv G\malce V$ consists of all finite semigroups $S$ admitting a homomorphism $\p\colon S\to T$ with $T\in \pv V$ and $\p\inv(e)\in \mathbb L\pv G$ for each idempotent $e$ of $T$.  Theorem~\ref{largeJ'} in conjunction with~\cite[Theorem 4.6.50]{qtheor} yields our next result.

\begin{Cor}
Let $\pv V$ be a pseudovariety of semigroups and let $S$ be a finite semigroup.  Then $S\in \mathbb L\pv G\malce \pv V$ if and only if the image of $S$ under every irreducible representation over $\mathbb B$ belongs to $\pv V$.
\end{Cor}

\subsection{Duality and simple boolean modules}
If $M$ is a finite join semilattice with identity, then it is automatically a complete lattice.   The dual semilattice $M^{op}$ is $M$ with the reverse ordering.  It can be identified with the collection of functionals $f\colon M\to \mathbb B$ with pointwise operations.  Indeed, it is well known~\cite[Chapter 9]{qtheor} that the functionals are given by choosing $m\in M$ and defining \[\p_m(n) = \begin{cases} 0 & n\leq m \\ 1 & n\nleq m\end{cases}\] and so $m\leq k$ if and only if $\p_k\leq \p_m$.

From now on we identify $M^{op}$ with the space of functionals.
It will be convenient to use the following boolean analogue of the Stone-Weierstrass theorem.

\begin{Prop}\label{stoneweier}
Let $M$ be a finite join semilattice with identity.  Then a subsemilattice with identity $N\subseteq M^{op}$ is equal to $M^{op}$ if and only if it separates the points of $M$.
\end{Prop}
\begin{proof}
By the above construction of the functionals, it is clear that $M^{op}$ separates the points of $M$.  Suppose that $N$ is a proper subsemilattice of $M$ with identity and let \[M'=\{m\mid \p_m\in N\}.\]  Then $M'$ is a proper meet-semilattice containing the top of $M$.  Let $m$ be a maximal element of $M$ that does not belong to $M'$.  Then the set of elements strictly above $m$ is non-empty and belongs to $M'$. It therefore has a meet $m'$ in $M'$, which must be the unique cover of $m$.  Suppose that $f\in N$.  We claim $f(m)=f(m')$.  Indeed, if $f=\p_k$ with $k\in M'$, then we have two cases.  If $m\leq k$, then since $m<k$ we have $m'\leq k$ and so $f(m)=0=f(m')$.  If $m\nleq k$, then certainly $m'\nleq k$ and so $f(m)=1=f(m')$.  It follows that $N$ does not separate points of $M$.
\end{proof}

Suppose now that $M$ is a finite right-$\mathbb BS$ module for a finite semigroup $S$.  Then $M^{op}$ is naturally a left $\mathbb BS$-module by putting $sf(m) = f(ms)$.

\begin{Thm}[Duality]
Let $S$ be a finite semigroup and let $M$ be a finite right $\mathbb BS$-module.  Then $M$ is simple (minimal) if and only if $M^{op}$ is minimal (simple).
\end{Thm}
\begin{proof}
Since $(M^{op})^{op}\cong M$, it suffices to show that $M$ is simple if and only if $M^{op}$ is minimal.  Suppose first that $M$ is simple and let $N$ be a non-zero $\mathbb BS$-submodule of $M^{op}$.  Define a congruence $\equiv$ on $M$ by $m\equiv m'$ if $f(m)=f(m')$ for all $f\in N$.  This is a congruence because $m\equiv m'$ and $s\in S$ implies that, for all $f\in N$, we have $f(ms) = sf(m)=sf(m')=f(m's)$ as $sf\in N$.  Thus $ms\equiv m's$ and so $\equiv$ is a congruence.  Since $N$ contains a non-zero functional $f$, it follows that $\equiv$ is a proper congruence.  Thus by simplicity of $M$, it follows that $\equiv$ is the trivial congruence.  Thus the elements of $N$ separate points and so $N=M^{op}$ by Proposition~\ref{stoneweier}. It follows that $M^{op}$ is minimal.

Conversely, suppose that $M^{op}$ is minimal and let $\equiv$ be a non-trivial congruence on $M$.  Let $N=\{f\in M^{op}\mid {\equiv}\subseteq \ker f\}$.  Then $N$ is a $\mathbb BS$-submodule of $M^{op}$.  Indeed, trivially $0\in N$. If $f,g\in N$ and $m\equiv m'$, then $(f+g)(m) = f(m)+g(m)= f(m')+g(m') = (f+g)(m')$.  Finally, if $f\in N$,  $s\in S$ and $m\equiv m'$, then  $ms\equiv m's$ and so $sf(m) = f(ms)=f(m's)=sf(m')$.  Thus $sf\in N$.  Since $M^{op}$ separates points and $\equiv$ is non-trivial, we conclude that $N\neq M^{op}$ and hence $N=0$ by minimality.  Suppose $\equiv$ is proper.  Then $M/{\equiv}\neq 0$ and has enough functionals to separate points.  Hence there is a non-zero functional $f$ on $M$ with ${\equiv}\subseteq \ker f$, a contradiction.  Thus $\equiv$ is not proper.  This completes the proof that $M$ is simple.
\end{proof}

It follows that in principle, we can construct the simple $\mathbb BS$-modules by dualizing the minimal left $\mathbb BS$-modules, which are constructed in Theorem~\ref{desribeirreducibles}.

\subsection{Characters}
In this subsection we propose a definition of the character of a boolean representation and also the notion of a generalized character.  These results are preliminary and will be expanded on in a future paper.

Let $M$ be a finite lattice, which we view as a $\mathbb B$-module via its join (and hence we utilize additive notation for the join).

\begin{Def}[\sji]
An element $m\in M$ is said to be \emph{strictly join irreducible} (or \sji) if $m\neq 0$ and $m=m_1+m_2$ implies $m=m_1$ or $m=m_2$.  See~\cite[Chapter 6.1.2]{qtheor}.
\end{Def}

 Let us denote by $\min M$ the set of \sji\ elements: it is the the unique minimal spanning set of $M$ as a $\mathbb B$-module. In particular, $M$ is a free $\mathbb B$-module if and only if $\min M$ is a basis.  Let us say that a decomposition
\begin{equation}\label{irredundant}
m=\sum_{x\in \min M}c_xx
\end{equation}
is \emph{irredundant} if changing any coefficient $c_x$ from a $1$ to a $0$ results in a strictly smaller element of $M$. In other words, an irredundant decomposition is one of the form $m=\sum_{x\in X} x$ with $X\subseteq \min M$ and such that for no proper subset $Y\subsetneq X$ is $m=\sum_{y\in Y} y$.  If $M$ is free, then any decomposition of an element $m$ is irredundant and in particular $x\in \min M$ appears in an irredundant decomposition of $m$ if and only if $x\leq m$.

\begin{Example}
Let $M$ be the span of $B=\{(1,0,0),(1,1,0),(0,1,1)\}$ in $\mathbb B^3$.  Then $\min M=B$.  Now one checks
\begin{align*}
(1,1,1) &= (1,0,0)+(0,1,1)\\
(1,1,1) &= (1,1,0)+(0,1,1)
\end{align*}
and these are all the irredundant decompositions of $(1,1,1)$.
\end{Example}

Suppose that $S$ is a finite semigroup and $M$ is a module for $S$ over a field $k$.  Let $\chi$ be the character of $M$.  Then if $B$ is a basis for $M$, one has that \[\chi(s) = \sum_{b\in B}c_{bb}\] where \[bs = \sum_{b'\in B}c_{bb'}b'.\]  In other words, it sums over all $b\in B$ the multiplicity of $b$ in the decomposition of $bs$ with respect to the basis $B$ (where the multiplicity is taken in the field $k$).  The problem with extending this idea to boolean case is the non-uniqueness of irredundant decompositions.  So instead we try and minimize over all decompositions.  With this in mind, we proceed to define the character of a $\mathbb BS$-module in two steps.

\begin{Def}[Character of a boolean matrix representation]
Suppose that $\p\colon S\to M_n(\mathbb B)$ is a matrix representation.  Let $B$ be the basis of $\mathbb B^n$.  Then define  \[\chi_{\p}(s) = \left|\{b\in B|bs\geq b\}\right|.\]  Equivalently, $\chi_{\p}(s)$ is the trace $\mathrm{tr}(\p(s))$ where we view $\p(s)$ as a zero-one matrix over $\mathbb C$.  If $M$ is the corresponding $\mathbb BS$-module, then we also use the notation $\chi_{M}$ for its character.
\end{Def}

Recall~\cite{berstelperrinreutenauer} that a boolean representation $\p\colon S\to M_n(\mathbb B)$ is \emph{unambiguous} if the product $\p(s)\p(t)$, viewed as matrices over $\mathbb C$, coincides with $\p(st)$. Unambiguous representations play a key role in the theory of rational codes~\cite{berstelperrinreutenauer}. We shall say that a $\mathbb BS$-module $M$ is \emph{unambiguous} if it is a free $\mathbb B$-module and the corresponding matrix representation is unambiguous. For instance, if $S$ acts on a finite set $X$, then the $\mathbb BS$-module $\mathbb BX$ is unambiguous.  From the definition, it is immediate that the character of an unambiguous module is a complex character of the semigroup.

Next suppose that $M$ is a finite $\mathbb BS$-module (with $S$ a finite semigroup).  Then there is a natural surjective $\mathbb B$-linear map $\pi\colon \mathbb B (\min M)\to M$ induced by the identity map on $\min M$.  Associated to each set-theoretic section $\sigma \colon M\to \mathbb B(\min M)$ of $\pi$ with $\sigma|_{\min M}=1_{\min M}$ is a $\mathbb BS$-module structure $M_{\sigma}$ on $\mathbb B (\min M)$ defined by $sx = \sigma(sx)$ for $s\in S$ and $x\in \min M$.  Moreover, notice that if $S$ has an identity $1$ and $M$ is unitary, then $1x=\sigma(1x)=\sigma(x)=x$ and so $M_{\sigma}$ is unitary.   In any case, one has that $\pi\colon M_{\sigma}\to M$ is a surjective morphism of $\mathbb BS$-modules.

\begin{Def}[Min character]
Let $S$  be a finite semigroup and $M$ a finite $\mathbb BS$-module.  Then the \emph{min character} $\chi_M\colon S\to \mathbb N$ is defined by \[\chi_M(s) = \min\{\chi_{M_{\sigma}}(s)\}\] where $\sigma$ runs over all set-theoretic sections of $\pi\colon \mathbb B(\min M)\to M$ with $\sigma|_{\min M}=1_{\min M}$.
\end{Def}

In more concrete terms, to compute $\chi_M(s)$, one first fixes an irredundant decomposition of each element of $M$.  Then one counts how many $x\in \min X$ appear in the chosen irredundant decomposition of $xs$.  Then one minimizes this quantity over all choices made.  Of course, if $M$ is a $\mathbb BS$-module that is free as a $\mathbb B$-module, then clearly the two notions of the character of $M$ coincide.

\begin{Example}
Let $M$ be the span of $B=\{(1,0,0),(1,1,0),(0,1,1)\}$ in $\mathbb B^3$.  The remaining elements of $M$ are $(0,0,0)$ and $(1,1,1)$.   As observed earlier $\min M=B$.  Any admissible section of $\pi$ must be the identity on all elements of $M$ except $(1,1,1)$, which has three lifts:
\begin{align*}
(1,1,1) &= (1,0,0)+(0,1,1)\\
(1,1,1) &= (1,1,0)+(0,1,1)\\
(1,1,1) &= (1,0,0)+(1,1,0)+(0,1,1)
\end{align*}
and so there are three corresponding sections $\sigma_1,\sigma_2,\sigma_3$.
Let $S$ be the two-element semilattice $\{1,e\}$ and let $S$ act on $M$ by having $1$ act as the identity and $e$ act via the map sending all non-zero elements to $(1,1,1)$ (and of course preserving zero).  Then $M$ is a $\mathbb BS$-module.  Let us compute $\chi_M(e)$.  It is easy to compute $\chi_{M_{\sigma_1}}(e) = 2$, $\chi_{M_{\sigma_2}}(e) = 2$ and $\chi_{M_{\sigma_3}}(e) = 3$.  Thus $\chi_M(e)=2$.  Notice that $\sigma_3$ corresponds to a redundant decomposition and hence has to be eliminated when taking the minimum.

Notice that if we were working over $\mathbb C$, then $B$ would be a basis for $\mathbb C$ and if $e$ took each basis vector to $(1,1,1)$, then since $(1,1,1) = (1,0,0)+(0,1,1)$ we would also get $2$ as being the character value on $e$.
\end{Example}

Next we discuss the notion of a generalized character.  Let $M$ be a finite $\mathbb BS$-module and $\min M\subseteq D\subseteq M\setminus \{0\}$.  Then we define the \emph{generalized character} $\psi_{M,D}\colon S\to \mathbb N$ by \[\psi_{M,D}(s) = \left|\{m\in D\mid ms=m\}\right|.\]  For example, if $S$ acts on the finite set $X$, then the generalized character $\psi_{\mathbb BX,X}$ is the complex character of the $\mathbb CS$-module $\mathbb CX$ and also coincides with the min character introduced above.  Let us define the \emph{generalized character spectrum} $\mathrm{cspec}(M)$ to be the set of all $\min M\subseteq D\subseteq M\setminus \{0\}$ such that $\psi_{M,D}$ is a complex character of $S$.

\begin{Prop}\label{specnonempty}
Let $M$ be a non-zero finite $\mathbb BS$-module.  Let $\min M\subseteq D\subseteq M\setminus \{0\}$ be a subset such that $d\in D$ implies $ds\in D\cup \{0\}$ for all $s\in S$. Then $D\in \mathrm{cspec}(M)$. In particular, $M\setminus \{0\}\in \mathrm{cspec}(M)$.
\end{Prop}
\begin{proof}
Let us write for the moment $\theta$ for the zero of $M$ and put $D'=D\cup \{\theta\}$.  Then $S$ acts by total functions on $D'$.    Thus $\mathbb CD'$ is a finite dimensional $\mathbb CS$-module in a natural way and $\mathbb C\theta$ is a $\mathbb CS$-submodule.  Consider the $\mathbb CS$-module $V=\mathbb CD/\mathbb C\theta$.  It is easy to see that $\psi_{M,D}$ is precisely the character of $V$.  This completes the proof.
\end{proof}

Recall that if $S$ is a semigroup and $R$ is an $\R$-class the Sch\"utzenberger representation of $S$ on $R$ is the actions of $S$ on $R$ by partial transformations given by $r\cdot s=rs$ if $rs\in R$ and otherwise is undefined (where $r\in R$ and $s\in S$).  One can turn $\mathbb CR$ into a $\mathbb CS$-module as follows.  First let $S$ act on $R\cup \{\square\}$ by sending all undefined products to $\square$ and demanding $\square s=\square$ for all $s\in S$.  Then consider the $\mathbb CS$-module $\mathbb C(R\cup \{\square\})/\mathbb C\square$ (which as a vector space is isomorphic to $\mathbb CR$).

\begin{Thm}
Let $S$ be a finite semigroup and $M_J$ an irreducible $\mathbb BS$-module where we retain the notation of Theorem~\ref{desribeirreducibles}.  Then there is subset $\min M\subseteq D\subseteq M\setminus \{0\}$ such that the generalized character $\psi_{M,D}$ is the complex character of $S$ obtained by lifting the right Sch\"utzenberger representation of $\mathsf {AGGM}_J(S)$ on an $\R$-class of its distinguished $\J$-class via the projection $\Gamma_J\colon S\to \mathsf{AGGM}_J(S)$.
\end{Thm}
\begin{proof}
Without loss of generality, we may assume that the action of $S$ on $M_J$ is faithful and hence $S=\mathsf{AGGM}_J(S)$. In this case, since $J$ is aperiodic, we may identify the action of $S$ on $B$ by partial transformations with the Sch\"utzenberger representation.  Let $[b]$ denote the equivalence class of $b\in B$ under the congruence $\equiv$ on $\mathbb B$ from Theorem~\ref{desribeirreducibles}.  Then $D=\{[b]\mid b\in B\}$ satisfies the conditions of Proposition~\ref{specnonempty}.  If we can show that the map $b\mapsto [b]$ is injective, then the proof of Proposition~\ref{specnonempty} will imply that $\psi_{M_J,D}$ is the complex character of the Sch\"utzenberger representation.    Suppose that $[b]=[b']$ with $b\neq b'$.  Let $a$ be an $\R$-class of $J$ and choose $s\in a\cap b$ and $s'\in a\cap b'$.  Then, for $x\in B$, one has $[x]s =[b]$ if $C_{xa}=1$ and is otherwise $0$ and $[x]s' = [b']$ if $C_{xa}=1$ and is otherwise $0$.  Thus $s$ and $t$ act the same on $M_J$ and so the action of $S$ on $M_J$ is not faithful, a contradiction.  This completes the proof.
\end{proof}

\section{Density}
This section relates our work with that of Zumbr{\"a}gel~\cite{Zsemiring}. In particular, we look at irreducible modules for semirings and discuss an application to semigroups.  Of course, any irreducible module for a semiring $R$ is an irreducible module for its underlying multiplicative semigroup, but the converse is false.  So in principle if $R$ is a finite semiring, then we can use our results to understand its irreducible representations as a semiring.  One just has to determine which irreducible representations preserve the additive structure.  However, since we are not assuming in general that the semirings in question are finite and also because proofs can become shorter and coordinate-free by taking advantage of the additive structure, we do not treat the representation theory of semirings as a special case of the representation theory of semigroups.

A semiring $R$ is called \emph{primitive} if it has a faithful irreducible module $M$.  Let $D$ be a division ring and $M$ a left vector space over $D$.  Then a subring $R\subseteq \mathrm{End}_D(M)$ is said to be \emph{dense} if, for any pair of $k$-tuples $(m_1,\ldots, m_k)$ and $(m_1',\ldots, m_k')$ of linearly independent elements of $M$, there is an element $r\in R$ with $(m_1r,\ldots, m_nr)=(m_1',\ldots, m_k')$.  If $R$ is a primitive ring with faithful irreducible module $M$, then $\mathrm{End}_R(M)$ is a division ring $D$ by Schur's lemma and so $R\subseteq \mathrm{End}_D(M)$.  Jacobson's density theorem shows that $R$ must be a dense subring.  Zumbr\"agal~\cite{Zsemiring} proved an analogous result for finite primitive semirings with idempotent addition, although he stated it only under the hypothesis that $R$ is \emph{congruence-simple}.  In fact, for finite semirings with idempotent addition, being primitive implies being congruence-simple.  For finite rings, primitivity is also equivalent to simplicity since a finite primitive ring $R$ will have a faithful finite irreducible module $M$ and hence $D$ above will be a finite division ring and thus a field by a theorem of Wedderburn.  But then $R$ will be a finite dimensional algebra over a field with a faithful irreducible module and hence by Wedderburn-Artin theory is simple.

Because Zumbr\"agal does not state his result~\cite{Zsemiring} in full generality, we reproduce it for the reader's convenience.  In what follows we assume that $M$ is a join semilattice with minimum and maximum.  For example any finite join semilattice with minimum also has a maximum.  We shall always denote the maximum by $\infty$.  Following  Zumbr\"agal~\cite{Zsemiring}, given $a,b\in M$, define \[xe_{a,b} = \begin{cases} 0 & x\leq a\\ b & \text{else.}\end{cases}\]  One can check that the $e_{a,b}\in \mathrm{End}(M)$ and that $\{e_{a,b}\mid a,b\in M\}\cup \{0\}$ is a subsemigroup $E$ since \[e_{a,b}e_{c,d} =\begin{cases} 0 & b\leq c\\ e_{a,d} & b\nleq c.\end{cases}\]  It is shown below that $M$ is an irreducible module for $E$. Zumbr\"agal defines $R\subseteq \mathrm{End}(M)$ to be \emph{dense} if it contains all the $e_{a,b}$ with $a,b\in M$.

\begin{Prop}\label{simpleimpliesprimitive}
Let $M$ be a join semilattice with minimum and let $E$ be as above.  Then $M$ is an irreducible $\mathbb BE$-module.  Thus any dense subsemiring of $\mathrm{End}(M)$ is primitive.
\end{Prop}
\begin{proof}
First observe that $M$ is minimal.  Indeed, if $0\neq m,n\in M$ then $me_{0,n}=n$.  Next suppose that $\equiv$ is a non-trivial semilattice congruence on $M$.  To show that $\equiv$ is not proper, suppose $m\equiv n$ with $m\neq 0$.  Then without loss of generality we may assume $m\nleq n$.   Let $a\in M$.  Then $me_{m,a} = 0$ and $me_{n,a}=a$.  Thus $0\equiv a$, completing the proof since $a$ was arbitrary.
\end{proof}

The following is~\cite[Theorem 2.3]{Zsemiring}.

\begin{Thm}[Zumbr\"agal]
Let $M$ be a join semilattice with minimum and let $R\subseteq \mathrm{End}(M)$ be dense.  Then $R$ is congruence-simple.
\end{Thm}

The next result is a strengthening of the statement of~\cite[Proposition 3.13]{Zsemiring}, but the proof is the same.  The reader should compare with Theorem~\ref{desribeirreducibles}.

\begin{Thm}\label{primitivegivesdense}
Let $R$ be an idempotent semiring and $M$ a faithful irreducible $R$-module with maximum.  Then $R$ is a dense subring of $\mathrm{End}(M)$.
\end{Thm}
\begin{proof}
Define a congruence on $M$ by $m\equiv n$ if $mr=0\iff nr=0$ for all $r\in R$.  This is easily verified to be a congruence.  Moreover, it cannot be universal since $m\equiv 0$ implies $mR=0$.  Thus it is the trivial congruence. Let $I_m$ denote the annihilator in $R$ of $m$.  It is a right ideal.

First note that $e_{\infty,b}=0$ and so belongs to $R$.  So fix $\infty\neq a\in M$.  We first show $e_{a,\infty}\in R$.    Suppose $x\nleq a$.  We claim $I_a$ does not annihilate $x$.  Indeed, if $I_a$ annihilates $x$, then since $(x+a)r=0$ if and only if $xr=0=ar$, it follows that $x+a\equiv a$ and so $x\leq x+a=a$.  Now $xI_a$ is then a non-zero submodule of $M$ and so $xI_a=M$.   Thus we can find $r_x\in I_a$ so that $xr_x=\infty$.  Let \[s=\sum_{x\nleq a}r_x\in I_a.\]  Then if $x\leq a$, one has $xs\leq as=0$, whereas if $x\nleq a$, then $xs\geq xr_x=\infty$ and so $s=e_{a,\infty}$.

Now let $\infty\neq b\in M$.  Then by irreducibility, $\infty R=M$ and so $b=\infty r$ for some $r\in R$.  Then, $e_{a,\infty}r = e_{a,\infty r} = e_{a,b}$ and so $R$ is dense, as required.
\end{proof}

As a consequence, we obtain the following result of~\cite{Zsemiring} (wherein the equivalence with the first item is not stated explicitly).

\begin{Cor}[Zumbr\"agal]
Let $R$ be a finite idempotent semiring with $R^2\neq 0$.  Then the following are equivalent:
\begin{enumerate}
\item $R$ is primitive;
 \item $R$ is isomorphic to a dense subsemiring of $\mathrm{End}(M)$ for some finite semilattice $M$ with minimum;
\item $R$ is congruence-free.
\end{enumerate}
\end{Cor}
\begin{proof}
The implication 2 implies 3 is~\cite[Theorem 2.3]{Zsemiring}, whereas 3 implies 1 is~\cite[Proposition 3.10]{Zsemiring}.  The final implication follows from Theorem~\ref{primitivegivesdense} and the observation that any irreducible module for a finite idempotent semiring must be finite and idempotent (since $M=mR$ for any non-zero element $m\in M$).
\end{proof}

It follows that if $S$ is a finite semigroup, $k$ is a finite idempotent semiring and $M$ is an irreducible $kS$-module, then the span of the image of the representation $\p\colon S\to \mathrm{End}_k(M)$ is dense.
Zumbr\"agal showed~\cite[Proposition 4.9, Remark 4.10]{Zsemiring} that if $M$ is finite, then $\mathrm{End}(M)$ contains no proper dense subsemirings if and only if $M$ is a distributive lattice.  In particular, this applies when $M$ is a finitely generated free $\mathbb B$-module since then $M$ is isomorphic to the power set of its basis.  We summarize this discussion in the next corollary.

\begin{Cor}
Suppose $S$ is a finite semigroup and let $\p\colon S\to \mathrm{End}(M)$ be an irreducible representation where $M$ is a join semilattice with identity.  Then $\p(S)$ spans a dense subsemiring of $\mathrm{End}(M)$.  In particular, if $M$ is a distributive lattice, then $\p(S)$ spans $\mathrm{End}(M)$.  Consequently, given an irreducible representation $\p\colon S\to M_n(\mathbb B)$, the image of $\p$ spans $M_n(\mathbb B)$.
\end{Cor}

The last part of the corollary can also be deduced in a straightforward way from Theorem~\ref{desribeirreducibles}.  The key step is to note that the span of the rows of $C$ is a free $\mathbb B$-module if and only if $C$ has an identity submatrix of the appropriate rank.

\def\malce{\mathbin{\hbox{$\bigcirc$\rlap{\kern-7.75pt\raise0,50pt\hbox{${\tt
  m}$}}}}}\def\cprime{$'$} \def\cprime{$'$} \def\cprime{$'$} \def\cprime{$'$}
  \def\cprime{$'$} \def\cprime{$'$} \def\cprime{$'$} \def\cprime{$'$}
  \def\cprime{$'$}

\end{document}